\documentclass[a4paper,twoside,reqno]{amsart}

\usepackage[pagewise]{lineno}
\usepackage[utf8]{inputenc}

\usepackage{authblk}
\usepackage{hyperref}
\hypersetup{urlcolor=blue, citecolor=red}
\usepackage{amsmath,amsthm,amssymb,xcolor,bbm,mathrsfs,graphicx, float}
\usepackage[english]{babel}
\usepackage{fancyhdr}
\DeclareMathOperator{\diag}{diag}

\makeatletter
\@namedef{subjclassname@1991}{2020 Mathematics Subject Classification}
\makeatother
\subjclass{91C20, 91D25, 94C15}
\keywords{Money exchange, social network, equal wealth, Fermi function}

\title{Application of the Fermi function in\\ money exchange}
\author{Hsin-Lun Li}

\date{}
\email{hsinlunl@asu.edu}

\theoremstyle{definition}
\newtheorem{theorem}{Theorem}

\newtheorem{lemma}[theorem]{Lemma}

\begin{document}

\allowdisplaybreaks

\thispagestyle{firstpage}
\maketitle
\begin{center}
    Hsin-Lun Li
    \centerline{$^1$National Sun Yat-sen University, Kaohsiung 804, Taiwan}
\end{center}
\medskip

\begin{abstract}
    In a money exchange process involving a seller and a buyer, we develop a straightforward model encompassing conservative, non-conservative, and systems with or without debt. Our model integrates the Fermi function to capture the behavior of buyers and sellers. Under certain circumstances, we identify an equation that marks the phase transition between a stable equal wealth state across all connected social graphs and an unstable equal wealth state in some connected social graph.
\end{abstract}

\section{Introduction}
A money exchange involves two parties: a seller and a buyer. The buyer pays the seller for goods or services, resulting in the seller's wealth increasing by the same amount that the buyer's wealth decreases. Money exchange models are typically classified as either conservative or non-conservative, and may also involve debt or not~\cite{yakovenko2009colloquium,lanchier2017rigorous,lanchier2018rigorous,lanchier2019rigorous}. Some models in opinion dynamics, such as the Deffuant model with attraction and repulsion, can be viewed as money exchange models~\cite{castellano2009statistical,lanchier2020probability,MR3069370,mHK,mHK2}. In this paper, we introduce a simple money exchange model, called Model T, that can represent both conservative and non-conservative systems, as well as systems with or without debt. Consider a simple undirected social graph \( G = ([n], E) \), where the vertex set \( [n] = \{1, \ldots, n\} \) represents a collection of agents, and each edge in the edge set \( E \) represents a social relationship. Two agents can transact if they are socially connected by an edge in the graph. $N_i= \{j \in [n] : (i,j) \in E\}$ denotes the collection of all social neighbors of individual $i$. The degree of vertex $i$ in graph $G$ at time $t$ is equal to the cardinality of $N_i$. \emph{Equal wealth} is achievable if all individuals' wealth approaches the same value over time. The Laplacian of \( G \), denoted as \( \mathscr{L} \), is defined as \( \mathscr{L} = \diag(|N_1|, \ldots, |N_n|) - A \), where \( A_{ij} = \mathbbm{1}\{(i,j) \in E\} \). 

The mechanism of Model T operates as follows: At each time step, a pair of socially connected agents, say agents \( i \) and \( j \), is uniformly selected to transact. The probability that agents \( i \) and \( j \) will act as seller and buyer, respectively, is defined by the Fermi function:
\[
Q_{ij}(t) = \frac{1}{1 + \exp\left[-\eta_t(w_i(t) - w_j(t))\right]},
\]
where \( \eta_t \in \mathbb{R} \) and \( w_i(t) \in \mathbb{R} \) is the wealth of agent \( i \) at time \( t \). The update of an individual's wealth for all \( i \in [n] \) is given by:
\[
w_i(t+1) = \begin{cases}
   \mu_i(t) w_i(t) & \text{if agent } i \text{ does not transact at time } t, \\
   \mu_i(t) (w_i(t) + k_t) & \text{if agent } i \text{ is the seller at time } t, \\
   \mu_i(t) (w_i(t) - k_t) & \text{if agent } i \text{ is the buyer at time } t,
\end{cases}
\]
where \( k_t \geq 0 \) is the transaction amount at time \( t \), and \( \mu_i(t) \in \mathbb{R} \) is a configuration influencing the wealth of agent \( i \) at time \( t \). For instance, a typhoon decreases local farmers' wealth, while winning a lottery increases the lottery winner's wealth. $w_i(t)<0$ indicates that agent $i$ is in debt at time $t.$ There are many factors affecting a person's wealth. Instead of listing all the factors, we can update individual \( i \)'s wealth with or without a transaction by multiplying it by \( \mu_i \).

Notice that \( Q_{ij}(t) + Q_{ji}(t) = 1 \), meaning that for any pair of socially connected agents \( i \) and \( j \), they will either be seller and buyer or buyer and seller, respectively. The role of \( \eta_t \) acts as follows:
\begin{itemize}
    \item \( \eta_t > 0 \) indicates that the wealthier agent is likelier to be a seller,
    \item \( \eta_t < 0 \) indicates that the wealthier agent is likelier to be a buyer,
    \item \( \eta_t = 0 \) indicates that both agents have an equal chance to be a seller or a buyer.
\end{itemize}
The wealthier agent is almost surely a seller if \(\eta_t \to \infty\) and a buyer if \(\eta_t \to -\infty\). Although Model T is simple, it is difficult to derive theoretical results since $\eta_t,\ \mu_i(t)$ and $k_t$ can vary over time. However, we can obtain theoretical results under the following assumption:
\begin{equation}\label{A1}
    \eta_t = \eta, \ \mu_i(t) = \mu \ \text{and} \ k_t = k \text{ for some nonzero constants } \eta, \mu \text{ and } k.
\end{equation}

\section{Main results}
Under assumption~\eqref{A1} and \(0 < |\mu| < 1\), we derive the inequality \( |\mu(1 + k\eta/2)| < 1 \), which ensures that all individuals in all connected social graphs reach the stable equal wealth state \( 0 \in \mathbb{R}^n \) as \( n \to \infty \). This inequality applies both when the wealthier agent is more likely to be a seller and when they are more likely to be a buyer. Additionally, there exists a connected social graph with an unstable equal wealth state if \( |\mu(1 + k\eta/2)| > 1 \). The condition \( |\mu| < 1 \) indicates that an individual's wealth decreases over time. Specifically, for individuals without debt, their wealth diminishes, while for those with debt, their debt decreases. Depreciation could be a possible condition that results in \( |\mu| < 1 \). We validate in Section~\ref{sec:Properties of Model T} that the stable equal wealth state is \( 0 \in \mathbb{R}^n \) for any initial wealth distribution. 

\begin{theorem}\label{Thm:critical upper bound} 
    Under assumption~\eqref{A1} and $|\mu|\in (0,1)$, $|1+k\eta/2|$ is a critical lower bound for $1/|\mu|$ to obtain the stable equal wealth state \(0 \in \mathbb{R}^n\) as \(n \to \infty\) for all connected social graphs, i.e.,
    \begin{itemize}
    \item $0 \in \mathbb{R}^n$ is stable as $n \to \infty$ for all connected social graphs if $|1+k\eta/2|<1/|\mu|$, and
    \item there is some connected social graph with an unstable equal wealth state as $n \to \infty$ when $|1+k\eta/2|>1/|\mu|$.
    \end{itemize}
\end{theorem}
The light-blue region in Figure~\ref{fig:stability} satisfies \( |\mu(1 + k\eta/2)| < 1 \). The equation \( |\mu(1 + k\eta/2)| = 1 \) separates the phase transition between a stable equal wealth state on all connected social graphs and an unstable equal wealth state on some connected social graph as $n\to\infty$. Observe that in both cases, where a wealthier agent is likelier to be a seller or a buyer, the larger the transaction amount is, the smaller \( |\mu| \) must be to achieve a stable equal wealth state on all connected social graphs. Moreover, the area of a stable equal wealth state is larger when $\eta < 0$ than when $\eta > 0$, which means that it is more likely to achieve a stable equal wealth state when the wealthier agent is more likely to be a buyer than a seller. We have demonstrated in Lemma~\ref{lemma:stability of an equal wealth society} that \( |\mu| \left(1 \vee \left|1 + \frac{k\eta}{2|E|}\lambda_1\right|\right) = 1 \) is the equation for each social graph that separates the phase transition between a stable equal wealth state and an unstable equal wealth state, where \(\lambda_1\) is the largest eigenvalue of the Laplacian of \(G\). This equation also depends on the social graph.

\begin{figure}[H]
    \centering
    \includegraphics[width=0.5\textwidth]{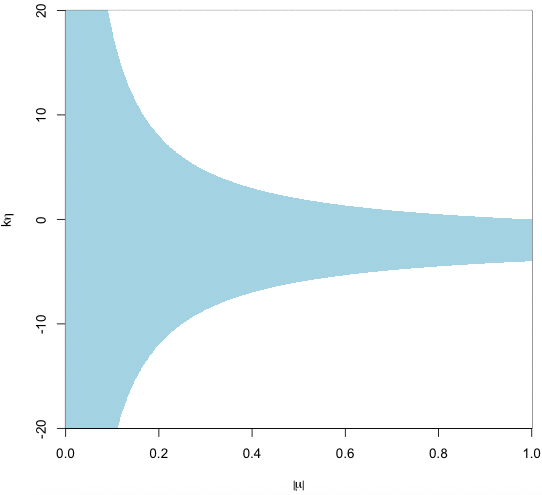}
    \caption{Region of $|\mu(1+k\eta/2)|<1$}
    \label{fig:stability}
\end{figure}

\section{Properties of Model T}
\label{sec:Properties of Model T}
\begin{lemma}\label{lemma:equal wealth}
    Under assumption~\eqref{A1}, all agents' wealth approaches $$\frac{\lim_{t\to\infty}\mu^t}{n}\sum_{i\in [n]}w_i(0)\ \hbox{if an equal wealth society is achieved.}$$
\end{lemma}
\begin{proof}
Letting $d_i=|N_i|$, $Q_{ij}=Q_{ij}(t)$, $w_i^\star=w_i(t+1)$ and $w_i=w_i(t).$ Via mean field approximation,
\begin{align*}
     w_i^\star&=(1-\frac{d_i}{|E|})\mu w_i+\frac{\mu}{|E|}\sum_{j\in N_i}\big[Q_{ij}(w_i+k)+Q_{ji}(w_i-k)\big]\\
    &=(1-\frac{d_i}{|E|})\mu w_i+\frac{\mu}{|E|}\sum_{j\in N_i}\big[w_i+k(2Q_{ij}-1)\big]\ \hbox{for all}\ i\in [n].
\end{align*}
Therefore, 
\begin{align*}
    \sum_{i\in [n]}w_i^\star&=\frac{\mu}{|E|}\sum_{i\in [n]}(|E|-d_i)w_i+\frac{\mu}{|E|}\bigg(\sum_{i\in [n]}d_iw_i+2k|E|-k2|E|\bigg)\\
    &=\frac{\mu}{|E|}\bigg(|E|\sum_{i\in [n]}w_i-\sum_{i\in [n]}d_iw_i+\sum_{i\in [n]}d_iw_i\bigg)=\mu\sum_{i\in [n]}w_i.
\end{align*}
Letting $W_t=\sum_{i\in [n]}w_i(t)/n$, we get
\begin{align*}
    W_{t+1}&=\mu W_t=\mu^{t+1}W_0\to W_0\lim_{t\to\infty}\mu^t\ \hbox{as}\ t\to\infty.
\end{align*}
This implies everyone's wealth approaches $$\frac{\lim_{t\to\infty}\mu^t}{n}\sum_{i\in [n]}w_i(0)\ \hbox{if an equal wealth society is achieved.}$$ 
\end{proof}

\begin{lemma}[\cite{das2003improved}]\label{lemma:largest eigenvalue of a Laplacian}
    The largest eigenvalue $\lambda_1$ of a Laplacian of the undirected simple graph $G=([n],E)$ satisfies $\lambda_1\leq \max\{|N_i\cup N_j|: (i,j)\in E\}.$
   
\end{lemma}

\begin{lemma}\label{lemma:eigenvalues of the Jacobian}
    Under assumption~\eqref{A1}, let $J$ be the Jacobian matrix of $w$ during an equal wealth society,  $\mathscr{L}$ be a Laplacian of graph $G$ and $a=\frac{k\eta}{2|E|}$. Then, $v$ is the eigenvector of $\mathscr{L}$ corresponding to eigenvalue $b$ $\iff$ $v$ is the eigenvector of $J$ corresponding to eigenvalue $\mu(1+ab)$.   
\end{lemma}

\begin{proof}
Letting $d_i=|N_i|$, $w_i=w_i(t)$, $w_i^\star=w_i(t+1)$ and $Q_{ij}=Q_{ij}(t).$ Under an equal wealth society, we derive
$\partial Q_{ij}/\partial w_i=\eta/4$ and $\partial Q_{ij}/\partial w_j=-\eta/4$, therefore
\begin{align*}
    &\frac{\partial w_i^\star}{\partial w_i}=(1-\frac{d_i}{|E|})\mu+\frac{\mu}{|E|}\bigg[1+2k\frac{\eta}{4}\bigg]d_i=\mu\bigg[1+\frac{k\eta d_i}{2|E|}\bigg],\\
    &\frac{\partial w_i^\star}{\partial w_j}=-\frac{\mu k\eta}{2|E|}\ \hbox{for}\ j\in N_i\ \hbox{and}\ 0\ \hbox{for}\ j\notin N_i\cup \{i\}.
\end{align*}
It turns out from $J_{ij}=\partial w_i^\star/\partial w_j$ that $$J-\lambda I=\mu(a\mathscr{L}+I)-\lambda I=\mu a\big[\mathscr{L}+\frac{1}{a}(1-\frac{\lambda}{\mu})I\big]$$
for $I\in\mathbb{R}^{n\times n}$ the identity matrix. Thus, $v$ is the eigenvector of $\mathscr{L}$ corresponding to eigenvalue $b=-(1-\lambda/\mu)/a$ $\iff$ $v$ is the eigenvector of $J$ corresponding to $\lambda=\mu(1+ab).$  
\end{proof}

Since a Laplacian is positive semidefinite, the dominant eigenvalue of $J$ is $|\mu|(1\vee|1+a\lambda_1|)$. Via stability test, we get Lemma~\ref{lemma:stability of an equal wealth society}.

\begin{lemma}\label{lemma:stability of an equal wealth society}
    Under assumption~\eqref{A1}, an equal wealth society is stable if $$|\mu|(1\vee|1+\frac{k\eta}{2|E|}\lambda_1|)<1,\quad \hbox{and unstable if}\quad |\mu|(1\vee|1+\frac{k\eta}{2|E|}\lambda_1|)>1.$$
\end{lemma}
It turns out that $|\mu|$ must be less than 1 if $|\mu|(1\vee|1+\frac{k\eta}{2|E|}\lambda_1|)<1.$ Via Lemma~\ref{lemma:equal wealth}, all agents' wealth approaches 0 over time.

\begin{proof}[\bf Proof of Theorem~\ref{Thm:critical upper bound}]
    Using Lemmas~\ref{lemma:equal wealth} and~\ref{lemma:stability of an equal wealth society}, \(0 \in \mathbb{R}^n\) is the stable equal wealth state if \(|\mu|(1 \vee |1+\frac{k\eta}{2|E|}\lambda_1|) < 1\). The condition \(|\mu|(1 \vee |1+\frac{k\eta}{2|E|}\lambda_1|) < 1\) implies that \( |\mu| < 1 \) and \( -\frac{1}{|\mu|} < 1 + \frac{k\eta}{2|E|}\lambda_1 < \frac{1}{|\mu|} \). From Lemma~\ref{lemma:largest eigenvalue of a Laplacian} and the assumption that the social graph is connected, it follows that

\[
\liminf_{n\to\infty}\left(1+\frac{k\eta}{2|E|}\lambda_1\right) \geq \liminf_{n\to\infty}\left(1+\frac{k\eta}{2(n-1)}n\right) = 1+\frac{k\eta}{2} \quad \text{if} \quad \eta < 0,
\]

\[
\limsup_{n\to\infty}\left(1+\frac{k\eta}{2|E|}\lambda_1\right) \leq \limsup_{n\to\infty}\left(1+\frac{k\eta}{2(n-1)}n\right) = 1+\frac{k\eta}{2} \quad \text{if} \quad \eta > 0.
\]

Observe that \(\lambda_1/|E| = n/(n-1)\) when the social graph is a star, so

\[
\lim_{n\to\infty}\left(1+\frac{k\eta}{2|E|}\lambda_1\right) = 1+\frac{k\eta}{2}.
\]

Thus, by Lemma~\ref{lemma:stability of an equal wealth society}, the proof is complete.

\end{proof}

\section{Statements and Declarations}
\subsection{Competing Interests}
The author is funded by NSTC grant.

\subsection{Data availability}
No associated data was used.

\end{document}